\documentclass[pdflatex,sn-mathphys-num]{sn-jnl}

\usepackage{graphicx}%
\usepackage{multirow}%
\usepackage{amsmath,amssymb,amsfonts}%
\usepackage{amsthm}%
\usepackage{mathrsfs}%
\usepackage[title]{appendix}%
\usepackage{xcolor}%
\usepackage{textcomp}%
\usepackage{manyfoot}%
\usepackage{booktabs}%
\usepackage{algorithm}%
\usepackage{algorithmicx}%
\usepackage{algpseudocode}%
\usepackage{listings}%

\usepackage{amsmath}
\usepackage{amsthm}
\usepackage{amssymb}

\usepackage[utf8]{inputenc}
\usepackage{hyperref}
\usepackage{subcaption}
\hypersetup{
	unicode,
	pdfauthor={Author One, Author Two, Author Three},
	pdftitle={A simple article template},
	pdfsubject={A simple article template},
	pdfkeywords={article, template, simple},
	pdfproducer={LaTeX},
	pdfcreator={pdflatex}
}

\theoremstyle{plain}
\newtheorem{theorem}{Theorem}
\newtheorem{corollary}[theorem]{Corollary}
\newtheorem{lemma}[theorem]{Lemma}

\theoremstyle{definition}
\newtheorem{definition}[theorem]{Definition}

\usepackage{graphicx, color}
\graphicspath{{../figures/}}

\usepackage{algorithm, algpseudocode} 
\usepackage{mathrsfs} 

\usepackage{makecell}

\usepackage{nicematrix}
\usepackage{blkarray}
\usepackage{xspace}

\newcommand{\bzero}{\Block{2-2}<\LARGE>{0}}

\newcommand{\ode}[1]{$\mathcal{O} ({#1})$}

\newcommand{\fpc}{Four-point-condition\xspace}
\newcommand{\sfpc}{Strong-four-point-condition\xspace}

\newcommand{\dd}{\ensuremath{\rho}\xspace}
\newcommand{\sd}{\ensuremath{\hat{\rho}}\xspace}
\newcommand{\centers}{\ensuremath{C}\xspace}
\newcommand{\patas}{\ensuremath{l}\xspace}
\newcommand{\vertebras}{\ensuremath{h}\xspace}

\def\metspace{\mathcal{X}}

\DeclareMathOperator*{\LC}{L}
\DeclareMathOperator*{\RC}{R}

\newcommand{\CC}{\ensuremath{\mathcal{C}}\xspace}
\newcommand{\QQ}{\ensuremath{\mathcal{Q}}\xspace}

\newcommand{\RR}{\ensuremath{\mathbb{R}}\xspace}
\newcommand{\xx}{\ensuremath{\mathbf{x}}\xspace}

\pacs[MSC Classification]{68R12, 90C90, 52-08}

\begin{document}

\title[Robinson spaces and their representation in low-dimensional metric spaces]{Robinson spaces and their representation in low-dimensional metric spaces}


\author*[1]{\fnm{Francisco} \sur{Arrepol} 
} \email{farrepol2016@udec.cl}
\equalcont{These authors contributed equally to this work.}

\author[2]{\fnm{Mauricio} \sur{Soto-Gomez}
}\email{mauricio.soto@unimi.it}
\equalcont{These authors contributed equally to this work.}

\author[3]{\fnm{Christopher} \sur{Thraves Caro} 
}\email{cthraves@udec.cl}
\equalcont{These authors contributed equally to this work.}

\affil[1]{\orgdiv{Departamento de Inform\'atica y Ciencias de la Computaci\'on, Facultad de Ingenier\'ia}, \orgname{Universidad de Concepci\'on}, 
\state{Concepci\'on}, \country{Chile}}

\affil[2]{\orgdiv{AnacletoLab, Department of Computer Science}, \orgname{University of Milan}, \orgaddress{\state{Milan}, \country{Italy}}}

\affil[3]{\orgdiv{ Departamento de Ingenier\'ia Matem\'atica, Facultad de Ciencias F\'isicas y Matem\'aticas}, \orgname{Universidad de Concepci\'on}, \orgaddress{\state{Concepci\'on}, \country{Chile}}}


\abstract{Robinson spaces are structures equipped with a total order that encodes comparative dissimilarity relationships. 
 We study the problem of representing Robinson dissimilarity spaces into low-dimensional metric spaces.
These representations aim to preserve the relative dissimilarity relationships between elements rather than their exact values. While low dimensional Euclidean spaces such as \( \mathbb{R}^1 \) and \( \mathbb{R}^2 \) are natural candidates for such embeddings, previous work has shown that not all Robinson spaces admit a valid embedding in the real line that respects their structural constraints. Motivated by this limitation, we explore the broader class of \emph{real trees}, which retain low-dimensional interpretability while allowing greater flexibility.

To address the embedding problem, we develop two key tools: a combinatorial representation of Robinson spaces and a topological characterization of \emph{caterpillars}, a restricted class of real trees. These tools enable a  formulation of the embedding problem as a linear program, providing both computational and theoretical insights.
We prove that some subclasses of Robinson spaces always admit embeddings in a caterpillar, and we establish the existence of Robinson spaces that cannot be embedded in any real tree. These results clarify the geometric limitations of representing ordered dissimilarity structures and open new directions for studying the interaction between dissimilarity, order, and metric geometry.}

\keywords{Robinson spaces, dissimilarity spaces, metric representation, valid drawing, caterpillars.}

\maketitle

\section{Introduction}\label{sec1}
\label{sec:intro}

A \emph{dissimilarity measure} between two objects is a numerical value that indicates how different the two objects are from one another: the smaller the value, the more similar the objects. In the case when the dissimilarity is zero, the two objects are identical. A \emph{dissimilarity space} consists of a set of elements along with a dissimilarity measure for each pair of distinct elements within that set. Dissimilarity representation is crucial in pattern recognition because it effectively captures structural and relational information between samples \cite{rezazadeh2021dissimilarity}.

In this work, we focus on \emph{Robinson spaces}, a special class of dissimilarity spaces characterized by the existence of a linear ordering of the elements  \(<\) such that, for any three elements \( x < y < z \), the dissimilarity between \( x \) and \( z \) is at least as large as both the dissimilarity between \( x \) and \( y \), and between \( y \) and \( z \). This combinatorial property allows to infer similarities directly from the ordering, and is key to the structural regularity of these spaces.

A dissimilarity space is not necessarily a metric space, as a dissimilarity measure need not satisfy all the properties of a distance (the triangle inequality, in particular, may not hold) even in Robinson spaces. However, representing dissimilarity spaces via the embedding of their elements into a metric space has proven useful in various fields, such as pattern recognition, and classification \cite{costa2020dissimilarity}.

The structural regularity of Robinson spaces leads us to conjecture that they can be efficiently represented in low-dimensional metric spaces. Since an exact preservation of dissimilarities is generally unattainable, we focus instead on preserving the \emph{relative} ordering induced by the dissimilarity values. That is, we aim to construct embeddings where, from the perspective of any given element, the closer of two others in the dissimilarity sense is also closer in the metric space. This kind of representation enables external observers to both visually and computationally recognize structural patterns and relationships present in the space.

In this context, a \emph{low-dimensional metric space} refers to a Euclidean space \(\mathbb{R}^k\), where \(k\) is small, typically \(k \in \{1, 2\}\). These low values of \(k\) facilitate geometric or visual interpretation of the space's structure. However, Aracena et al. in~\cite{weighted-line-scfe} showed that not all Robinson spaces admit an embedding in \(\mathbb{R}^1\) that preserves the desired relational structure, which motivates the exploration of more flexible low-dimensional settings.

A natural generalization of \(\mathbb{R}^1\) is the class of \emph{real trees}, which are connected, geodesic metric spaces without cycles. These spaces preserve the notion of low dimensionality in terms of Hausdorff dimension (see~\cite{makoto2005hausdorff, Schleicher01062007}) while enabling a wider variety of embeddings. Importantly, they retain interpretability and structure, making them suitable for representing hierarchical or branching relationships.

In this work, we study the problem of finding an embedding of the elements of a Robinson space into metric spaces while preserving the underlying relational structure, i.e., the comparative dissimilarities encoded by a compatible linear order. We show that this problem can be formulated as a linear program, and prove that specific subclasses of Robinson spaces always admit such embeddings. However, we also show that there exist Robinson spaces that do not admit any embedding of this kind into a real tree.

\section{Definitions, notation, and our results}
\label{sec:definitions}

Let \( S \) be a finite set. A \emph{dissimilarity} on \( S \) is a function \( \dd: S \times S \rightarrow \mathbb{R}_{\geq 0} \) such that \( \dd(x, y) = \dd(y, x) \) for all \( x, y \in S \), and \( \dd(x, y) = 0 \) if and only if \( x = y \). The pair \( (S, \dd) \) is called a \emph{dissimilarity space}. We say that a dissimilarity space \( (S, \dd) \) is \emph{strict} if \(\dd(x,y)\neq \dd(u,v)\) for any two distinct pairs \(\{x,y\}\neq\{u,v\}\), not both in the diagonal (i.e., \(x\neq y\) or \(u\neq v\)).

\begin{definition}[Robinson Space]\label{def:RobinsonSpace}
    A dissimilarity space \( (S, \dd) \) is \emph{Robinson} if there exists a total order \( < \) on \( S \) such that, for every triple \( x < y < z \) in \( S \), the following inequality holds:
    \[
        \dd(x, z) \geq \max\{\dd(x, y), \dd(y, z)\}.
    \]
    Such an order is called a \emph{compatible order}.
\end{definition}

In the following, we assume that finite Robinson spaces are of the form \( S=\{1, 2, \dots, n\} \) and that the identity order \( 1 < 2 < \cdots < n \) is compatible.

To geometrically represent dissimilarity spaces, we use the notion of a \emph{valid drawing}, which refers to an embedding into a metric space that preserves the relative similarity structure among the elements.
\begin{definition}[Valid Drawing]\label{def:ValidDrawing}
  Let \( (S, \dd) \) be a dissimilarity space and \( (\metspace, d) \) be a metric space.
  An injection \( I: S \rightarrow \metspace \) is a \emph{valid drawing} if for all \( x, y, z \in S \) such that \( \dd(x, y) < \dd(x, z) \), it holds that
    \begin{equation}
      \label{eq:valid_drawing}      
        d(I(x), I(y)) < d(I(x), I(z)).
    \end{equation}
\end{definition}

A widespread intuition on Robinson spaces is that their elements can be represented by points on a line \cite{Carmona2023}.
It is known that if a dissimilarity space admits a valid drawing in the real line, it is Robinson. 
However, not every Robinson space admits a valid drawing in \( \mathbb{R}^1 \)~\cite{weighted-line-scfe}. 
To overcome this limitation, we seek richer metric spaces that allow more flexibility in representing such structures without sacrificing interpretability. One natural choice is the class of \emph{real trees}.
\begin{definition}[Real Tree]
    A metric space \( (T, d) \) is a \emph{real tree} if it is path-connected, and for every triple \( x, y, z \in T \), there exists a point \( c \in T \) (called the center) such that the geodesics between the pairs \( (x, y) \), \( (y, z) \), and \( (z, x) \) intersect at \( c \).
\end{definition}
Real trees generalize the linear structure of \( \mathbb{R}^1 \), allowing the representation of a broader class of dissimilarity spaces. In particular, since Robinson spaces come with an inherent order, we focus on structured real trees that preserve this order, known as \emph{caterpillars}. 

\paragraph{Caterpillars.}
\label{sec:caterpillar}
Quoting the evocative description by Harary et al.~\cite{HARARY1973359}, ``a caterpillar is a tree which metamorphoses into a path when its cocoon of endpoints is removed.''
In our setting, the continuous caterpillar structure can be naturally interpreted as a discrete one, since we will embed a finite set 
\(S\) into it. This leads us to work with a weighted caterpillar graph, where weights encode the distances between adjacent vertices.

We adopt the following terminology for caterpillar graphs (see Figure~\ref{fig:caterpillar} for an illustration). The central path of a caterpillar is called the \emph{spine}. The vertices belonging to the spine are called \emph{spine vertices}. An edge whose both endpoints lie on the spine is called a \emph{spine edge}. Edges that connect a spine vertex to a leaf are called \emph{legs}. The leaf vertices attached to the spine through legs are simply called \emph{leaves}. This terminology will allow us to distinguish clearly between the structure of the central path and the peripheral attachments, and it will be used consistently throughout the remainder of the paper.

These structured trees strike a balance between expressive power and interpretability. Their linear backbone naturally reflects the total order inherent to Robinson spaces, making them compelling candidates for valid drawings that preserve comparative dissimilarities.
\begin{figure}[t]
  \centering
  \includegraphics[]{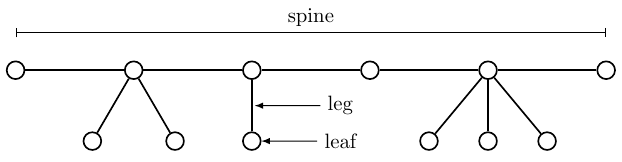}
  \caption{Example of a caterpillar graph}
  \label{fig:caterpillar}
\end{figure}

%
\paragraph{Left-Right Center.}
\label{sec:left-right-center}
Consider the set \(S=\{1,2,3,\ldots,n\}\), and let \((S,\dd)\) be a Robinson space for which the identity order \(1 < 2 < \cdots < n\) is compatible.
Let \( i \leq j \leq k \) be three elements of \( S \) and suppose that
\(\dd(i,j) < \dd(j,k)\),
that is, \( j \) is \emph{less dissimilar} from \( i \) than from \( k \). Since \( S \) is a Robinson space, the dissimilarities along the compatible order are monotone. Hence, for every \( j' \in S \) with \( i \leq j' \leq j \), we have
\[
\dd(i,j') \leq \dd(i,j) < \dd(j,k) \leq \dd(j',k),
\]
which implies
\(\dd(i,j') < \dd(j',k)\).
In other words, all elements between \( i \) and \( j \) are less dissimilar to \( i \) than to \( k \).

Analogously, if
\(\dd(i,j) > \dd(j,k)\),
meaning that \( j \) is less dissimilar to \( k \) than to \( i \), then for every \( j' \) with \( j \leq j' \leq k \) it holds that
\[
\dd(i,j') > \dd(j',k),
\]
so all elements between \( j \) and \( k \) are less dissimilar to \( k \) than to \( i \).

These observations show that the comparison between \( \dd(i,j) \) and \( \dd(j,k) \) is completely determined by the boundary where this inequality of dissimilarities changes. More precisely, it suffices to know, in the compatible order of \( S \), the last element that is less dissimilar to \( i \) than to \( k \), and the first element that is less dissimilar to \( k \) than to \( i \). This transition point captures all the information needed to compare \( \dd(i,j) \) and \( \dd(j,k) \).

The previous discussion motivates the definition of the following central notions. 

For each pair \( i < k \) in \( S \), the \emph{left center} \( \LC(i, k) \) is the largest index \( j \in [i, k] \) such that \( j \) is more similar to \( i \) than to \( k \):
\[
\LC(i,k) = \max \{ j \in [i, k] : \dd(i, j) < \dd(j, k) \}.
\]
Similarly, the \emph{right center} \( \RC(k, i) \) is the  smallest index \( j \in [i, k] \) such that \( j \) is more similar to  \( k \) than to \( i \):
\[
\RC(k,i) = \min \{ j \in [i, k] : \dd(i, j) > \dd(j, k) \}.
\]
We extend the definitions to the diagonal by setting \( \LC(i,i) = \RC(i,i) = i \) for all \( i \in S \).
These notions enable us to encode the relational structure of dissimilarities into a single matrix, defined as follows.
\begin{definition}[Matrix of Centers]
The \emph{matrix of centers} of a Robinson space \( (S, \dd) \) with \( n \) elements is the \( n \times n \) matrix \( \mathbf{\centers(S)}\) defined by:
\[
\mathbf{\centers(S)}_{(i,j)}=
\begin{cases}
    \LC(i,j) & \text{if } i<j,\\
    i       & \text{if } i=j,\\
    \RC(i,j) & \text{if } i>j.\\
\end{cases}
\]
\end{definition}
Figure \ref{fig:StrictRobinsonexample} shows an example of a strict Robinson space and its matrix of centers.

\begin{figure}[t]
    \centering
    \includegraphics[]{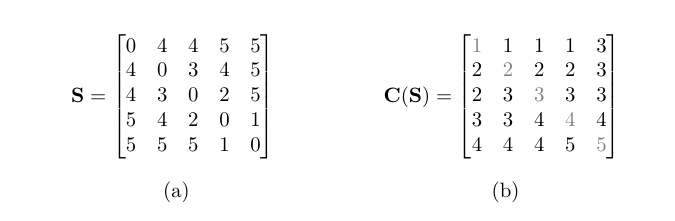}
    \caption{(a) Example of a Robinson space represented by its dissimilarity matrix \textbf{S}. (b) Matrix of centers \textbf{C(S)} of the space \textbf{S}.}
    \label{fig:StrictRobinsonexample}
\end{figure}

We point out an important structural property of the matrix of centers of a Robinson space:
For every Robinson space, the entries of $\mathbf{\centers(S)}$ exhibit monotone behavior along both rows and columns.
More precisely, for any fixed index $i$, the entries in row $i$ of $\mathbf{\centers(S)}$ are non-decreasing as the column index increases.
Similarly, for any fixed index $j$, the entries in column $j$ of $\mathbf{\centers(S)}$ are non-decreasing as the row index increases.
This  monotonicity follows directly from the definitions of left and right centers combined with the Robinson property, which enforces monotone structure with respect to the compatible order.

\paragraph*{Our contributions and structure of the document.}
The structure of the paper and our main contributions are organized as follows. We start by contextualizing our work within the existing literature in Section~\ref{sec:prevworks}.

In Section~\ref{sec:tree_valid_drawing}, we present a \emph{mixed integer linear programming} formulation for finding valid drawings of Robinson spaces into real trees. Although algorithmically useful, this approach may be computationally demanding and does not directly address the more fundamental question of whether all Robinson spaces admit such embeddings.

In Section~\ref{sec:strictrobinsonian}, we introduce a reduction that maps any Robinson space to a strict Robinson space with the property that any valid drawing of the strict space also yields a valid drawing of the original space. This reduction allows us to restrict our attention to strict Robinson spaces without loss of generality.

In Section~\ref{subsec:sfpc}, we introduce the \sfpc, which characterizes exactly those metric spaces that admit isometric embeddings into caterpillars (Theorem~\ref{thm:sfpc_caterpillar}), a result of independent interest. Building on this characterization, in Corollary~\ref{thm:caterpillar} we show that any strict Robinson space admitting a valid drawing in a real tree must necessarily admit one in a caterpillar. In Section~\ref{sec:lp_formulation_strict}, we exploit these structural insights to achieve two objectives: first, we derive an efficient linear programming formulation to decide whether a strict Robinson space admits a valid drawing in a tree; and second, we establish a framework for identifying structural properties that guarantee or preclude the existence of such drawings (Lemma~\ref{thm:kernel_characterization}).

In Section~\ref{sec:characterization-applications}, we apply these tools to obtain both positive and negative results. On the positive side, in Theorem~\ref{thm:center-feasible} we show that any strict Robinson space satisfying
\[
\LC(i,j) \le \tfrac{i+j}{2} \le \RC(j,i) \quad \text{for all } i<j
\]
admits a valid drawing in a caterpillar, providing an easily verifiable sufficient condition. In Theorem~\ref{thm:single-non-redundant}, we establish an alternative sufficient condition based on a more compact representation of the associated linear program. Moreover, we use our computational tools to verify empirically that all Robinson spaces on at most five elements admit valid drawings in trees.

On the negative side, in Theorem~\ref{thm:novalidtree}, we exhibit a strict Robinson space that cannot be embedded into any caterpillar, thereby proving that some Robinson spaces are fundamentally incompatible with tree-based representations that preserve their comparative dissimilarity structure.

Finally, in Section~\ref{sec:conclusions}, we discuss future directions and conclude the paper.

\section{Related work}
\label{sec:prevworks}

Robinson spaces were introduced by W.S. Robinson in \cite{seriation}, where he studied how to order chronologically archeological deposits. In the same study, Robinson introduced the \textit{seriation} problem, which consists of deciding whether a dissimilarity space is Robinson or not, and finding a compatible order if possible. Several works have focused on the recognition and structural
characterization of Robinson spaces. Fortin and Préa \cite{algorithm-robinson}
presented an \ode{n^2} optimal recognition algorithm for recognizing Robinson spaces based on PQ-Trees,
and interval graph properties. Alternative approaches were later proposed by
Laurent and Seminaroti using Lex-BFS \cite{lex-bfs-robinson}, whose time complexity is \ode{L(m+n)}, where $m$ is the number of pairs in the space whose dissimilarity is non-zero, and $L$ is the number of different dissimilarity values in the space. Later, Laurent and Seminaroti generalized this method using
similarity first search \cite{laurent-similarity}, resulting in a recognition algorithm with time complexity \ode{n^2 + nm \log n}. Further structural insight was provided by Laurent et al. \cite{laurent-robinson-characterization},
who characterized Robinson spaces through the absence of weighted asteroidal
triples. More recently, Carmona et al. \cite{modules-robinson} introduced the
notion of modules, and used this notion to propose a divide-and-conquer recognition algorithm with
\ode{n^2} complexity.

The notion of valid drawings was introduced by
Kermarrec and Thraves in \cite{Thraves-2011} as the \textit{Sitting Closer to Friends than Enemies} (SCFE) problem, focused on signed graphs. In this study, the authors provided a polynomial algorithm to determine whether a complete signed graph has a valid drawing on the line, and characterized the set of complete signed graphs that have a valid drawing on the line. A subsequent work by
Cygan et al. \cite{re-thraves-2011} established a connection between valid
drawings in the real line and proper interval graphs, and showed that deciding
the existence of such drawings is NP-complete when considering incomplete signed graphs. The SCFE problem was also explored as an optimization problem by Pardo et al. \cite{pardo-thraves-soto}, who related the problem of finding a valid drawing on the line to the quadratic assignment problem, and provided two optimization algorithms to construct valid drawings while trying to minimize the number of errors, that is, the violation of the inequality that defines a valid drawing.

The SCFE problem has also been studied for different metric spaces. Spaen et
al. in \cite{Spaen2020TheDO} studied the problem of finding the minimal dimension $k$ required such that any signed graph of $n$ vertices has a valid drawing in $\mathbb{R}^k$, denoted by $L(n)$, and proved that $log_5(n-3) \leq L(n) \leq n-2$. Benítez et al. in \cite{scfte-circle} studied the SCFE problem in the circumference, showing that it is NP-Complete, and characterized the complete signed graphs that have a valid drawing on the circumference using proper circular arc graphs. Becerra and
Thraves \cite{scfe-tree} analyzed valid drawings of signed graphs in real trees, showing that a
complete signed graph admits such a drawing if and only if its positive subgraph
is strongly chordal.

An extension of valid drawings was proposed by Aracena and
Thraves in \cite{weighted-line-scfe}, where they studied the existence of a valid
drawing on the real line for weighted graphs. They showed that the SCFE problem for dissimilarity spaces on the line is related to the seriation problem, and they also proved that the two problems are not equivalent. They proved that the Robinson property is necessary, but not sufficient, for the existence of valid drawings on the line, and that if a complete Robinson space has a valid drawing on the line for a compatible order, then it has a valid drawing for any of its compatible orders. They also formulated the weighted SCFE problem as an optimization problem, and provided a polynomial time method to determine if a complete dissimilarity space has a valid drawing on the line.

\section{Tree valid drawing problem}
\label{sec:tree_valid_drawing}

  In this section, we present a Mixed Integer Linear Program (MILP) that, given a Robinson space, determines whether it admits a valid drawing in a real tree. The construction of this MILP relies on a structural criterion to characterize tree metrics, which can be obtained from the following results. 

\begin{definition}[\fpc]
  Let $d$ be a metric on a set of objects $\metspace$. We say that $d$ satisfies the \fpc if, for all $\{x, y, u, v\} \subseteq \metspace$, it holds:
  \begin{equation}
    \label{eq:fpc}
    d(x, y)+d(u, v) \leq \max \{ d(x, u)+d(y, u), d(x, v)+d(y, u) \}.
  \end{equation}
\end{definition}

  This condition is key to ensuring the existence of a tree representing a finite metric space. 

  \begin{theorem}[Buneman~\cite{Buneman1974}]
  Given an $(n \times n)$ distance matrix $\mathbf{D}$.
  There is an unrooted tree whose path metric is compatible with $\mathbf{D}$ if and only if the metric defined by $\mathbf{D}$ satisfies \fpc.
\end{theorem}

  Moreover, the tree associated with a metric satisfying the \fpc is unique~\cite{Hendy92}.

  \paragraph{MILP formulation.}
  A common technique for the construction of a tree metric meeting a set of constraints is the use of a linear programming formulation~\cite{WSSB77,CohenAddad2024,FORTZ17}.
  
  Given a Robinson space $(S,\dd)$ we define, for each pair of distinct elements $i, j \in S$, the variable $d_{ij}$ representing the distance between the images of $i$ and $j$ in a tree valid drawing of $S$, with $d_{ii}=0$ for each $i\in S$.
  
  From this set of variables, a potential formulation for the tree valid drawing of a Robinson space is presented in \ref{milp1}, where variables must satisfy three types of constraints.

  First, \textbf{valid drawing conditions} enforce the distances to respect the restrictions imposed by the compatible order (Equation~\ref{eq:valid_drawing}).
  Note that these constraints require strict inequalities; therefore, in \ref{milp1} we introduce a small constant $\varepsilon >0$ to represent them as inclusive inequalities.
    
  Second, \textbf{metric conditions} must guarantee that variables define a metric space, that is, distances between elements must be non-negative, symmetric, and must satisfy the triangle inequality. 

  Third, the \textbf{\fpc} ensures that the induced metric is compatible with a tree metric (Equation~\ref{eq:fpc}).
  These restrictions are non-convex and require the use of binary indicator variables.
  In \ref{milp1} this condition is modeled using the big-$M$ method (where $M\gg 0$) together with binary variables $x_{ijkl}^1,x_{ijkl}^2$ and $x_{ijkl}^3$ for each  $\{i,j,k,l\} \in \QQ$, where $\QQ$ denotes the set of (unordered) quadruplet having distinct values.
  These decision variables identify which of the three possible pairwise sums in Equation~\ref{eq:fpc} is minimal, forcing the two larger sums to be equal.

  {\small  \begin{align*}
  \label{milp1} \tag{$LP1$}
      \text{minimize} \quad &  f(S,\dd)  \\
    \text{subject to} \quad
    & d_{ij} \le  d_{ik}+\varepsilon \hspace{85pt} \forall \,i,j,k \in S,\, \dd(i,j) < \dd(i,k)\\
                            & d_{ik} \le d_{ij}+\varepsilon \hspace{85pt} \forall \,i,j,k \in S,\, \dd(i,j) > \dd(i,k)\\[7pt]
    & d_{ik} \le d_{ij}+d_{jk} \hspace{140pt} \forall \,i,j,k \in S\\
    & d_{ij} = d_{ji}, \: d_{ij} \ge 0  \hspace{130pt} \forall\, \{i,j\}  \in S \\[5pt]
    & \hspace{-4pt}\left|\left(d_{i k}+d_{j l}\right)-\left(d_{i l}+d_{j k}\right)\right| \leq M\left(1-x_{i j k l}^1\right) \qquad \forall \,\{i,j,k,l\} \in \QQ\\
    & \left(d_{i j}+d_{k l}\right)- \left(d_{i k}+d_{j l}\right) \leq M\left(1-x_{i j k l}^1\right) \qquad \forall \,\{i,j,k,l\} \in \QQ\\
    & \left(d_{i j}+d_{k l}\right)-\left(d_{i l}+d_{j k}\right) \leq M\left(1-x_{i j k l}^1\right) \qquad \forall \,\{i,j,k,l\} \in \QQ \\[5pt]
   & \hspace{-4pt}\left|\left(d_{i j}+d_{k l}\right)-\left(d_{i l}+d_{j k}\right)\right| \leq M\left(1-x_{i j k l}^2\right) \qquad \forall \,\{i,j,k,l\} \in \QQ\\
   & \left(d_{i k}+d_{j l}\right) - \left(d_{i j}+d_{k l}\right) \leq M\left(1-x_{i j k l}^2\right) \qquad \forall \,\{i,j,k,l\} \in \QQ\\
   &  \left(d_{i k}+d_{j l}\right)-\left(d_{i l}+d_{j k}\right) \leq M\left(1-x_{i j k l}^2\right) \qquad \forall \,\{i,j,k,l\} \in \QQ \\[5pt]
   & \hspace{-4pt}\left|\left(d_{i j}+d_{k l}\right)-\left(d_{i k}+d_{j l}\right)\right| \leq M\left(1-x_{i j k l}^3\right)\qquad \forall \,\{i,j,k,l\} \in \QQ\\
   &  \left(d_{i l}+d_{j k}\right) - \left(d_{i j}+d_{k l}\right) \leq M\left(1-x_{i j k l}^3\right) \qquad \forall \,\{i,j,k,l\} \in \QQ\\
                            &  \left(d_{i l}+d_{j k}\right) - \left(d_{i k}+d_{j l}\right) \leq M\left(1-x_{i j k l}^3\right) \qquad \forall \,\{i,j,k,l\} \in \QQ\\
    & \sum_{s\in \{1,2,3\} }x_{i j k l}^s=1 \hspace{115pt} \forall \,\{i,j,k,l\} \in \QQ \\[5pt]  
    & x_{i j k l}^s \in \{0,1\} \hspace{70pt} \forall s \in \{1,2,3\}, \,\forall \,\{i,j,k,l\} \in \QQ
\end{align*}
}
  Note that the formulation \ref{milp1} is a feasibility problem. Therefore, the objective function can be chosen arbitrarily.
  Nonetheless, the formulation requires \( O(n^4) \) integer variables and constraints, which can be computationally demanding for large instances.
  In the following sections, we show that, for the class of strict Robinson spaces, tree valid drawings have a caterpillar topology. 
  This structural insight enables us to design a significantly more compact linear formulation.

\section{Reduction to strict Robinson spaces}
\label{sec:strictrobinsonian}

  In the following, we focus on the class of strict Robinson spaces. Recall that a Robinson dissimilarity space $(S,\dd)$ is said to be \emph{strict} if all its dissimilarity values are pairwise distinct.

Although strict Robinson spaces form a proper subclass of Robinson spaces, we show that this restriction is made without loss of generality. Indeed, every Robinson space can be associated with a strict Robinson space defined on the same set of elements and satisfying the following properties:
\begin{itemize}
  \item[(i)] the compatible order of the original space is preserved, and
  \item[(ii)] every valid drawing of the strict Robinson space is also a valid drawing of the original Robinson space.
\end{itemize}
Consequently, by working with strict Robinson spaces, we retain full generality while benefiting from stronger structural properties that simplify both the analysis and the formulation of the problem.

  The construction that maps a Robinson space $(S,\dd)$ to a strict Robinson space $(S,\sd)$ reassigns dissimilarity values to each pair according to the relative order of the original dissimilarities.
To define the strict Robinson space $(S,\sd)$, let us consider the set
\(\mathcal U = \bigl\{ \{i,j\} \subset S : i<j \bigr\}\)
of unordered pairs of elements of $S$.

First, we construct a total order $\pi$ on the set $\mathcal U$ according to the following sequence of criteria:
\begin{enumerate}
  \item the value of $\dd(i,j)$, in increasing order;
  \item the value of $i$, in decreasing order;
  \item the value of $j$, in increasing order.
\end{enumerate}
These criteria eliminate all potential ties among elements of $\mathcal U$, thus ensuring that $\pi$ is a total order.
Second, for each pair $(i,j)\in \mathcal U$, dissimilarity is defined as its rank according to order $\pi$:  
\[\sd(i,j) = \pi(i,j).\]
By construction, all values of $\sd(i,j)$ are pairwise distinct and belong to the set
\(\{1,2,\ldots,n(n-1)/2\}\).

Finally, we define the dissimilarity space $(S,\sd)$, which we call the \emph{strict mapping} of $(S,\dd)$, by symmetrically extending $\sd$ to all pairs of elements of $S$:
\[
\sd(i,i)=0 \quad \text{for all } i\in S,
\qquad
\sd(j,i)=\sd(i,j) \quad \text{for all } j > i.
\]
Clearly, the strict mapping  $(S,\sd)$ is a strict Robinson space on the same ground set $S$.
Figure~\ref{fig:strict-mapping} depicts an example of a Robinson space 
and its strict mapping. 

\begin{lemma} \label{lema: relacion-dibujos-construccion}
  Let $(S,\dd)$ be a Robinson space, and let $(S,\sd)$ be its strict mapping. Then, the following statements hold: 
\begin{itemize}
\item Every compatible order for $(S,\dd)$ is also compatible for  $(S,\sd)$.
    \item 
  Let $(\metspace,d)$ be a metric space, and suppose that $I:S\to\metspace$ is a valid drawing of $(S,\sd)$.
  Then $I$ is also a valid drawing of $(S,\dd)$.
  \end{itemize}
\end{lemma}
 \begin{proof} 
 We first show that every compatible order for $(S,\dd)$ is also compatible for $(S,\sd)$. 
To this end, let $i<j<k$ be three elements of $S$ ordered according to a compatible order for $(S,\dd)$. 
By definition of a Robinson space, we have
\[
\dd(i,k) \ge \max\{\dd(i,j),\dd(j,k)\},
\]
that is, $\dd(i,k)\ge \dd(i,j)$ and $\dd(i,k)\ge \dd(j,k)$.

By construction of the strict mapping, the order $\pi$ ranks pairs primarily according to increasing values of $\dd$, and breaks ties using the pair indices. Since $(i,k)$ has a dissimilarity value greater than or equal to those of $(i,j)$ and $(j,k)$, and because $i<j<k$, the rules defining $\pi$ ensure that
\[
\sd(i,k) >_{\pi} \sd(i,j)
\qquad\text{and}\qquad
\sd(i,k) >_{\pi} \sd(j,k).
\]
Consequently,
\[
\sd(i,k) > \max\{\sd(i,j),\sd(j,k)\},
\]
which proves that the order $i<j<k$ is also compatible for $(S,\sd)$. Hence, the same order is also compatible for $(S,\sd)$.

  Now, for the second statement, let  $I$ be a valid drawing of $(S,\sd)$.
  To prove that $I$ is also a valid drawing of $(S,\dd)$, consider $i,j,k\in S$ such that $\dd(i,j) < \dd(i,k)$.

  By the definition of the dissimilarity $\sd$ we have that
  $\sd(i,j) < \sd(i,k)$.
  Therefore, $d\bigl(I(i),I(j)\bigr) < d\bigl(I(i),I(k)\bigr)$.
  Thus, for every triplet of element $i,j,k \in S$
\[
\dd(i,j) < \dd(i,k)
\quad \Rightarrow \quad 
d\bigl(I(i),I(j)\bigr) < d\bigl(I(i),I(k)\bigr).
\]
Hence, $I$ is a valid drawing for $(S,\dd)$.
\end{proof}

\begin{figure}[t!]
\centering
\includegraphics[]{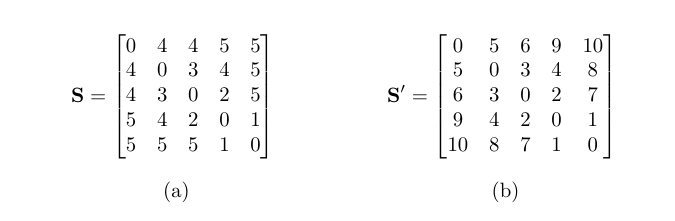}
\caption{(a) Example of a Robinson space \textbf{S} represented by its dissimilarity matrix. (b) The strict mapping for the space \textbf{S} represented by its dissimilarity matrix.}
\label{fig:strict-mapping}
\end{figure}

\section{\sfpc}\label{subsec:sfpc}
In this section, we characterize caterpillar metric spaces in the spirit of Buneman’s Theorem by providing a condition similar to the classical \fpc, which we call the \emph{\sfpc}.

Let $(S,\dd)$ be a strict Robinson space on $n>4$ elements, and consider any quadruple of vertices ordered according to a compatible order. For the sake of simplicity (and \emph{w.l.o.g.}), we denote these vertices by $1,2,3,$ and $4$. Any valid drawing of $(S,\dd)$ in a tree must satisfy the \fpc. Moreover, since $(S,\dd)$ is strict Robinson, we have
\(
\dd(1,2) < \dd(1,3)\)
and 
\(\dd(3,4) < \dd(2,4)\).
This implies that in any corresponding tree drawing, the distances satisfy
\[
d(1,2) < d(1,3)
\qquad\text{and}\qquad
d(3,4) < d(2,4).
\]
Consequently,
\[
d(1,2) + d(3,4) < d(1,3) + d(2,4).
\]

The \fpc implies
\begin{eqnarray*}
d(1,2) + d(3,4)&\leq&
\max\{ d(1,3) + d(2,4),\, d(1,4) + d(2,3)\},
\\
d(1,3) + d(2,4)&\leq&
\max\{ d(1,2) + d(3,4),\, d(1,4) + d(2,3)\},
\, \text{and}\\
d(1,4) + d(2,3)&\leq&
\max\{ d(1,3) + d(2,4),\, d(1,2) + d(3,4)\}.
\end{eqnarray*}
Together with the strict inequality
\(
d(1,2) + d(3,4) < d(1,3) + d(2,4),
\)
these relations force
\[
d(1,3) + d(2,4) = d(1,4) + d(2,3).
\]
This observation motivates the following definition:
  
  \begin{definition}[\sfpc]
    \label{def:sfpc}
    Let $(\metspace,d)$ be a finite metric space.
    We say that $d$ satisfies the \emph{\sfpc} if $d$ satisfies the \fpc and there exists an ordering $\pi$ of $\metspace$ such that for all $\{i, j, k, l\} \subseteq \metspace$ with $i \le_{\pi} j \le_{\pi} k \le_{\pi} l$, it holds:
  \begin{equation}
    \label{eq:sfpc}
d(i, j) + d(k, l) \leq d(i, k)+d(j, l) = d(i, l) + d(j, k).
  \end{equation}
\end{definition}

This result has important consequences for the search for valid drawings for Robinson spaces. 
First, from an algorithmic perspective, the condition significantly simplifies the constraints describing the \fpc. Indeed, for any quadruple, it is known in advance which of the three possible sums is the smallest one, and therefore no binary decision variables are required to encode this choice. 

Second, from a geometric point of view, in Theorem~\ref{thm:sfpc_caterpillar} we prove that tree metrics satisfying the \sfpc necessarily have a caterpillar topology. 
Together, these properties allow us to reduce both the size of the formulation and, ultimately, the computational complexity of finding a valid drawing in a tree.

 \begin{theorem}
      \label{thm:sfpc_caterpillar}
  Let \((\metspace, d)\) be a finite metric space. Then, $(\metspace,d)$ satisfies the \sfpc if and only if $(\metspace,d)$  can be isometrically embedded in a caterpillar. 
 \end{theorem}
 
 \begin{proof}
Let $(\metspace,d)$ be a metric space satisfying the \sfpc. 
We begin by showing that $(\metspace,d)$ admits an isometric embedding into a caterpillar. 
Assume that the elements of $\metspace$ are labeled according to an order $\pi$ prescribed by the \sfpc.

We proceed by induction on $n$. Assume that there exists a caterpillar compatible with the metric induced by the first $n$ elements of $\metspace$. We show that this caterpillar can be extended to incorporate the $(n+1)$-th element in $\pi$.

  \begin{figure}[t]
    \centering
      \includegraphics[]{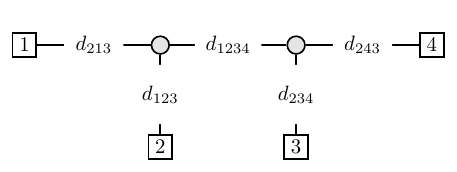}
  \caption{Caterpillar construction for the metric space $(\{1,2,3,4\},d)$ satisfying the \sfpc. The distances $d_{ijk}$ and $d_{ijkl}$ in the caterpillar are defined using the distance $d$ in the metric space as follows: $d_{ijk}=\frac{1}{2}(d(i,j)+d(j,k)-d(i,k))$, and $d_{ijkl}=d(i,l)+d(j,k)-(d(i,j)+d(k,l))$.}
  \label{fig:centipede_4}
\end{figure}

  In the base case, consisting of four elements (say \(1\leq_\pi 2\leq_\pi 3 \leq_\pi 4\), where \(\pi\) is the ordering given by the \sfpc), the caterpillar can be constructed as illustrated in Figure~\ref{fig:centipede_4}, where the length of each edge is given explicitly. In the figure, $d_{ijk}$ denotes the Gromov product, defined by
\[
d_{ijk}=\tfrac{1}{2}\bigl(d(i,j)+d(j,k)-d(i,k)\bigr),
\]
and we set
\[
d_{ijkl}=d(i,k)+d(j,l)-(d(i,j)+d(k,l)).
\]

All edge lengths in the construction are non-negative. Indeed, by the triangle inequality we have $d_{ijk}\ge 0$ for all triples $(i,j,k)$, and by the \sfpc we obtain
\[
d_{1234}=d(1,3)+d(2,4)-(d(1,2)+d(3,4)) \ge 0.
\]
Furthermore, by a direct algebraic verification, all pairwise distances realized in the caterpillar are seen to coincide with the corresponding values of the metric $d$.
 Hence, the caterpillar provides an isometric embedding of the four-point metric space.

 By the inductive hypothesis, assume that every metric space on $n$ objects satisfying the \sfpc\ can be isometrically embedded into a caterpillar. 
Let $(\metspace,d)$ be a metric space on $(n+1)$ objects satisfying the \sfpc, and consider the caterpillar in which its first $n$ objects (according to the order $\pi$ prescribed by the \sfpc) are embedded.

Notice that, in this caterpillar, the length of the spine edge that is incident to the vertices corresponding to $n-1$ and $n$ is given by $d_{1,n,n-1}$ (see Figure~\ref{fig:centipede_construction}). 
To extend the construction so as to include the element $n+1$, we subdivide this edge by inserting a new vertex $s$, and connect $s$ to a new leaf, which represents the image of $n+1$ in the tree. 
The lengths of the newly created edges are defined analogously to the base case and are illustrated in Figure~\ref{fig:centipede_construction}.

As before, by direct algebraic verification,  all newly created edge lengths are non-negative. 
It remains to check that the distances induced by the caterpillar are compatible with the metric $d$. 
Since all distances between pairs of vertices among $\{1,\ldots,n\}$ are preserved by the inductive hypothesis, it suffices to verify the distances involving the new vertex $n+1$.

For the pair $(n,n+1)$, we have
\[
d_{(n-1)n(n+1)} + d_{(n-1)(n+1)n} = d((n-1),n).
\]
For every $1 \le i \le n-1$, the distance in the caterpillar between $i$ and $n+1$ is:
\[
d(i,n) - d_{(n-1)n(n+1)} + d_{(n-1)(n+1)n}
= d(i,n) - d((n-1),n) + d((n-1),(n+1)),
\]
which coincides with $d(i,(n+1))$ by the \sfpc. 
Therefore, all pairwise distances involving the new vertex $n+1$ are correctly realized, and the extended caterpillar provides an isometric embedding of $(\metspace,d)$.

  In the opposite direction, assume that the metric space admits an isometric embedding into a caterpillar. 
Since the metric induced by a caterpillar is a tree metric, it necessarily satisfies the \fpc. 
Therefore, in order to show that a metric compatible with a caterpillar satisfies the \sfpc, it suffices to provide an ordering of its elements for which condition~\eqref{eq:sfpc} holds.

\begin{figure}[t]
  \centering
   \includegraphics[]{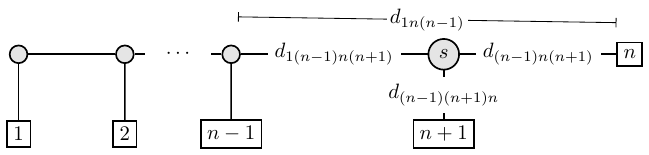}
  \caption{Construction of a caterpillar for a metric satisfying the \sfpc.}
  \label{fig:centipede_construction}
\end{figure} 

Consider the order $\pi$ induced by the spine of the caterpillar when traversed from one of its ends. 
More precisely, a vertex $i \in \metspace$ precedes a vertex $j$ in $\pi$ if the closest spine vertex to $i$ (possibly $i$ itself) appears before the closest spine vertex to $j$. Ties are broken arbitrarily. 
Consider any quadruple of elements of $\metspace$, \emph{w.l.o.g.}, \(1\leq_\pi 2\leq_\pi 3 \leq_\pi 4\).

As observed before, the distance between vertex $3$ and its closest neighbor on the spine is equal to $d_{234}$ (see Figure~\ref{fig:centipede_4}). 
This distance coincides with $d_{134}$. Hence, since
\(d_{234} = d_{134}\) we have:
\[
d(2,3)+d(3,4)-d(2,4)
=
d(1,3)+d(3,4)-d(1,4)\]
which implies \[
d(1,4)+d(2,3) = d(1,3)+d(2,4).
\]

Moreover, the distance between the spine vertices of vertices $3$ and $4$ is given by
\(d_{1234} \ge 0\)
(see Figure~\ref{fig:centipede_4}). 
Therefore,
\[
0 \le d_{1234}
= d(1,4) + d(2,3) - ( d(1,2) + d(3,4) )\]
which implies
\[
d(1,2) + d(3,4) \le d(1,4) + d(2,3).
\]
Combining this inequality with the equality obtained above, we conclude that
\[
d(1,2) + d(3,4)
\le d(1,4) + d(2,3)
= d(1,3) + d(2,4),
\]
which is precisely the condition required by the \sfpc.
\end{proof}
From the preceding discussion, any valid drawing of a strict Robinson space in a tree must satisfy the \sfpc. We thus obtain the following result.
\begin{corollary}
\label{thm:caterpillar}
If a strict Robinson space \((S,\dd)\) has a valid drawing in a tree \(T\), then \(T\) is a caterpillar. 
\end{corollary}

\section{Feasibility problem formulation}
\label{sec:lp_formulation_strict}
Using the characterization presented in the above section, in this section we present a feasibility formulation that allows us to decide whether a strict Robinson space admits a valid drawing in a caterpillar.

We use Corollary~\ref{thm:caterpillar} to formulate the problem of finding a valid drawing in a tree as a feasibility problem over a system of linear constraints. 
To this end, we introduce the following sets of variables, which represent the distances in a caterpillar (see Figure~\ref{fig:centipede}):

\begin{itemize}
  \item $\vertebras_i$ denotes the distance along the spine between the first spine vertex and the $i$-th spine vertex.
  \item $\patas_i$ denotes the length of the leg hanging from the $i$-th spine vertex.
\end{itemize}

Let $(\metspace,d)$ be a metric space on $n$ elements satisfying the \sfpc, and let $\pi$ be the corresponding order. 
Consider a caterpillar with $n$ leaves, ordered according to the appearance of the point where they hang, when the spine is traversed from one of its ends. 
If we inject the $i$-th element of $\metspace$ (with respect to $\pi$) into the $i$-th leaf of the caterpillar, then, for any $i<j$, the distance between $i$ and $j$ in the caterpillar is given by
\[
d(i,j)=\vertebras_j-\vertebras_i + \patas_j +\patas_i.
\]

\begin{figure}[t]
  \centering
    \includegraphics[]{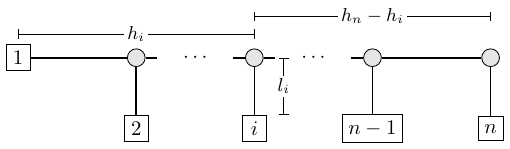}
    \caption{Variable definitions describing a caterpillar graph.}
  \label{fig:centipede}
\end{figure}

 We may assume  $\vertebras_1=\patas_1=0$.

  Moreover, since $(S,\dd)$ is a strict Robinson space, we have $\dd(1,i)<\dd(1,i+1)$ for all $1\le i <n$, and therefore $d(1,i)<d(1,i+1)$ in any valid drawing.
  This implies that, in a caterpillar drawing of $(S,\dd)$, elements must appear at strictly increasing distance from element $1$ according to the compatible order.
  A symmetric argument applies when considering distances from element $n$.
  Consequently, the conditions for a valid  drawing in a caterpillar need only be verified using the left and right centers of each pair.
  Specifically, for each pair $i,j\in S$ with $1\le i < j \le n$, we must have
\begin{equation*}
  d(i,\LC(i,j)) < d(\LC(i,j),j)
  \qquad\text{and}\qquad
  d(i,\RC(j,i)) > d(\RC(j,i),j).
\end{equation*}

  In terms of the variables $\vertebras_i$ and $\patas_i$, these conditions are equivalent to
\begin{equation}
\label{eq:centipede_conditions}
\vertebras_i + \vertebras_j - 2\vertebras_{\LC(i,j)} - \patas_i + \patas_j > 0,
\qquad
-\vertebras_i - \vertebras_j + 2\vertebras_{\RC(j,i)} + \patas_i - \patas_j > 0.
\end{equation}

Note that if a vector $\xx=(\mathbf{\vertebras},\mathbf{\patas})^\top=(\vertebras_2,\ldots,\vertebras_n,\patas_2,\ldots,\patas_n)^\top \in \mathbb{R}^{2n-2}_+$
satisfies the system \eqref{eq:centipede_conditions}, then the vector  $\alpha\mathbf{x}+\beta\mathbf{1}=(\alpha \vertebras_2+\beta,\ldots,\vertebras_n+\beta,\alpha \patas_2+\beta,\ldots,\alpha \patas_n+\beta)^\top$ also satisfies the same system, for every $\alpha\in\RR^+$ and $\beta\in\RR$.
 This invariance property allows us to simplify the formulation by removing the non-negativity constraints on the variables. 
Indeed, given any solution (possibly with negative entries) satisfying \eqref{eq:centipede_conditions}, one can add a sufficiently large constant $\beta$ to all variables in order to obtain a non-negative solution.

Therefore, the existence of a valid drawing of a strict Robinson space in a caterpillar is equivalent to showing that the following set is non-empty:
\begin{equation}\label{eq:CS}
    \CC_S=\!\left\{ \xx=(\mathbf{\vertebras},\mathbf{\patas})^\top \!\! \in \mathbb{R}^{2n-2}:
    \hspace{-8pt}
    \begin{array}{l@{\hskip 0.03cm}l}
        \phantom{-}\vertebras_i + \vertebras_j - 2\vertebras_{\LC(i,j)} -\patas_i+\patas_j &>0 \\
        -\vertebras_i - \vertebras_j + 2\vertebras_{\RC(j,i)} +\patas_i-\patas_j &> 0
    \end{array}\!
    , \forall \, 1\le i < j \le n \right\}.
\end{equation}

\paragraph{Matrix formulation.} \label{sec:matrix_formulation}
Since the set $\CC_S$ is defined by linear constraints, it can be equivalently described in matrix form as $\mathbf{A_S}\,\xx>0$ for some matrix $\mathbf{A_S}\in \RR^{n(n-1)\times (2n-2)}$.
The columns of $\mathbf{A_S}$ are indexed by the variables $\vertebras_k$ and $\patas_k$, with $k\in\{2,\ldots,n\}$, while the rows correspond to the constraints and are indexed by the ordered pairs $(i,j)\in [n]^2$ with $i\neq j$.

  By the definition of $\CC_S$, in the constraint indexed by a pair $(i,j)$ the only variables that may have a nonzero coefficient are those associated with $i$, $j$, and with the left or right center of the pair. A special situation occurs when the left or right center coincides with $i$ or $j$, that is, when $\LC(i,j)=i$ or $\RC(j,i)=i$. In these cases, the inequalities in \eqref{eq:centipede_conditions} simplify to
$-\vertebras_i + \vertebras_j - \patas_i + \patas_{j}  > 0 $ when $i<j$, and $\vertebras_i - \vertebras_j - \patas_i + \patas_{j} > 0 $  when $i>j$.
  Figure \ref{fig:matrix_A} depicts the matrix $\mathbf{A_S}$ for the example presented in Figure~\ref{fig:strict-mapping}.

This matrix formulation allows us to derive an alternative characterization of strict Robinson spaces admitting a valid drawing in a caterpillar. The characterization relies on the classical theorem of alternatives due to Gordan.

\begin{theorem}[Gordan's alternative {\cite{gordan1873ueber}}]
  \label{thm:gordan}
Let $\mathbf{A}$ be an $m\times n$ real matrix. Exactly one of the following statements holds:
\begin{enumerate}
    \item There exists $\mathbf{x}\in\mathbb{R}^n$ such that $\mathbf{Ax}>0$.
    \item There exists a nonzero vector $\mathbf{y}\in\mathbb{R}^m$ such that $\mathbf{A}^\top \mathbf{y}=0$ and $\mathbf{y}\ge 0$.
\end{enumerate}
\end{theorem}

In our setting, Theorem~\ref{thm:gordan} can be restated as follows. 
Let $(S,\dd)$ be a strict Robinson space, then exactly one of the following statements holds:
\begin{enumerate}
\item[\emph{1'.}] $(S,\dd)$ admits a valid drawing in a caterpillar. 
\item[\emph{2'.}] There exists a nonzero vector $\mathbf{y}\in \RR^{n(n-1)}$ such that $\mathbf{A_S}^\top \mathbf{y}=0$ and $\mathbf{y} \ge 0$.
\end{enumerate}
In other words, a strict Robinson space $(S,\dd)$ admits a valid drawing in a caterpillar if and only if the kernel of $\mathbf{A_S}^\top$ contains no nonzero nonnegative vector.

\begin{figure}[t]
  \centering
  \includegraphics[]{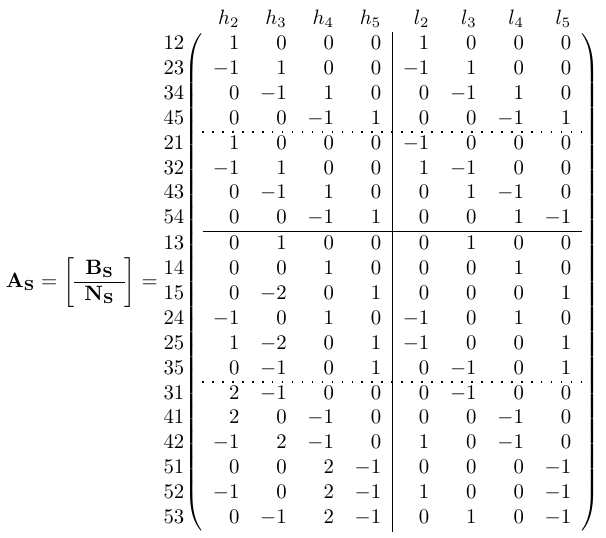}
 \caption{Matrix definition of the problem for the strict Robinson space example of Figure \ref{fig:strict-mapping} (b). In the figure, matrix rows represent the constraints induced by conditions \ref{eq:centipede_conditions}.}
  \label{fig:matrix_A}
\end{figure}

\paragraph{Kernel characterization.}
\label{sec:kernel_characterization}
 Consider now a partition of the rows (constraints) of the matrix $\mathbf{A_S}$ into two submatrices $\mathbf{B_S}$ and $\mathbf{N_S}$, where $\mathbf{B_S}$ consists of the rows indexed by consecutive pairs,
\[
\{12,\ldots,(n-1)n,\,21,\ldots,n(n-1)\},
\]
 and $\mathbf{N_S}$ contains the remaining constraints (see Figure~\ref{fig:matrix_A}). 
For consecutive pairs we have $\LC(i,i+1)=i$ and $\RC(i+1,i)=i+1$. 
Consequently, the submatrix $\mathbf{B_S}$ has the block structure shown in Figure \ref{eq:tn}.
\begin{figure}[t]
    \centering
    \includegraphics[]{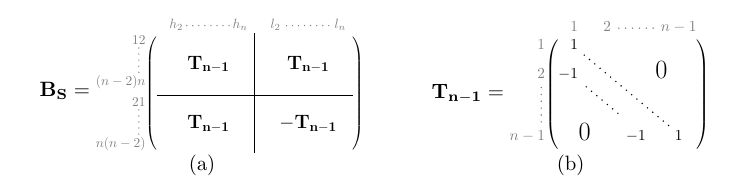}
    \caption{(a) Structure of matrix $\mathbf{B_S}$. (b) Structure of matrices $\mathbf{T_{n-1}}$.}
    \label{eq:tn}
\end{figure}
We emphasize that $\mathbf{B_S}$ depends only on $n$ and is independent of the specific instance $(S,\dd)$. 
Moreover, $\mathbf{B_S}$ has full rank and is therefore invertible.

 Let us consider a vector $\mathbf{y} \in \ker(\mathbf{A_S}^\top)$ and partition it as
\[
\mathbf{y}^\top = (\mathbf{y_B}^\top \mid \mathbf{y_N}^\top),
\]
according to the block decomposition of $\mathbf{A_S}$ into $\mathbf{B_S}$ and $\mathbf{N_S}$. 

  With this partition, every vector in $\ker(\mathbf{A_S}^\top)$ is completely determined by its components in $\mathbf{y_N}$. 
Indeed, the equation $\mathbf{A_S}^\top \mathbf{y} = 0$ can be written as
\[
\mathbf{B_S}^\top \mathbf{y_B} + \mathbf{N_S}^\top \mathbf{y_N} = 0,
\]
or, equivalently,
\[
\mathbf{y_B} = \mathbf{Y_S}\, \mathbf{y_N},
\]
where
\[
\mathbf{Y_S} := -(\mathbf{B_S}^\top)^{-1} \mathbf{N_S}^\top.
\]

To compute $\mathbf{Y_S}$, we observe that the inverse of the matrix $\mathbf{B_S}^\top$ admits the explicit expression
  \[
  \begingroup
  \renewcommand{\arraystretch}{1.5} 
  (\mathbf{B_S}^\top)^{-1}=\frac12
  \left(
    \begin{array}{c|r}
    \mathbf{T_{n-1}}^{-\top} & \mathbf{T_{n-1}}^{-\top}           \\
      \hline
    \mathbf{T_{n-1}}^{-\top} & -\mathbf{T_{n-1}}^{-\top}
    \end{array}
  \right)
  \endgroup
,
  \text{where}\quad 
  \mathbf{T_{n-1}}^{-\top} = 
  \resizebox{!}{.07\linewidth}{$
  \NiceMatrixOptions{xdots/shorten = 0.6 em,
    code-for-first-row = \color{gray},
    code-for-first-col = \color{gray},}
  \begin{pNiceArray}{cccc}[]

    1               & \Cdots &   & 1            \\

    \phantom{-1}    & \Ddots &   & \Vdots       \\

    \bzero          &        &   &              \\

                    &        &   & 1            \\ 
  \end{pNiceArray}
 $ }.
\]

Figure~\ref{fig:B_example} depicts the matrix $\mathbf{Y_S}$ for the example in Figure \ref{fig:strict-mapping}.
We summarize the above discussion in the following lemma.
  
\begin{lemma}\label{thm:kernel_characterization}
  A strict Robinson space $(S,\dd)$ on $n$ elements admits a valid drawing in a caterpillar if and only if the following system is infeasible:
  \begin{equation}
    \label{eq:lpS}\tag{LP-S}
    \mathbf{y_N}\in \RR^{(n-1)(n-2)}\colon 
    \quad
    \mathbf{Y_S}\, \mathbf{y_N}\ge 0,\quad \mathbf{y_N}\ge 0 ,\quad \mathbf{y_N}\neq 0
  \end{equation}

\end{lemma}

  Note that Lemma~\ref{thm:kernel_characterization} provides an equivalent description of the set $\CC_S$ defined in \eqref{eq:CS}. 
However, it yields a more explicit structural characterization of strict Robinson spaces that admit a valid drawing in a caterpillar. 

 \begin{figure}[t]\centering
\includegraphics[]{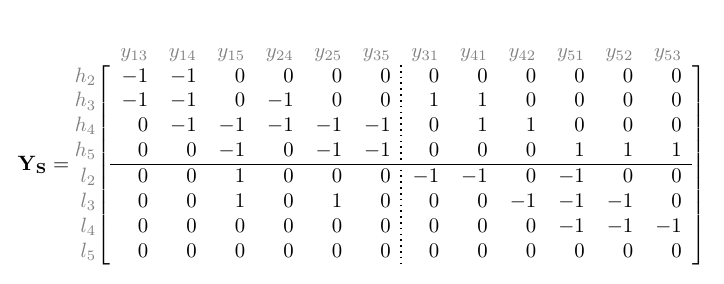}
\caption{Matrix $Y_S$ for the strict Robinson space example of Figure~\ref{fig:strict-mapping} (b).}
\label{fig:B_example}
\end{figure}

\section{Applications of Lemma \ref{thm:kernel_characterization}}
In this section, we present three applications of Lemma \ref{thm:kernel_characterization} that yield general results concerning the problem of finding a valid drawing in a tree for Robinson spaces.

  \paragraph*{Redundant triples.}
  \label{sec:redundant-triples}

Recall that the vector $\mathbf{y_N}$ is indexed by the constraints of the matrix $\mathbf{A_S}$ associated with pairs $(i,j)$ such that $|i-j|>1$. 
For any pair $1 \le i < j \le n$, the corresponding column of $\mathbf{Y_S}$ depends on the type of constraint represented by $(i,j)$. 
In fact, there are exactly four possible column structures, which are displayed in Figure~\ref{fig:B_sctructute}.

 We say that the column $(\mathbf{Y_S})_{ij}$, corresponding to the constraint associated with the pair $(i,j)$, is \emph{redundant} if imposing the condition ${\mathbf{y_N}}_{ij}=0$ does not affect the feasibility of problem \eqref{eq:lpS}. 
That is, the column $(\mathbf{Y_S})_{ij}$ is redundant if whenever there exists a solution to \eqref{eq:lpS}, there also exists a solution $\tilde{\mathbf{y}}_{\mathbf{N}}$ such that $\tilde{\mathbf{y}}_{\mathbf{N}ij}=0$.

\begin{table}[t]

\begin{tabular}{l c c c c c c c }
\hline
$n$  & 100 & 200 & 300 & 400 & 500 & 600 & 700 \\
$\hat{\mu}$ & 21.86\% & 20.86\% & 20.70\% & 20.49\% & 20.43\% & 20.20\% & 20.43\% \\
$\hat{\sigma}/\hat{\mu}$ & 8.78\% & 6.59\% & 5.57\% & 4.84\% & 4.36\% & 4.08\% & 3.56\% \\
\hline
\end{tabular}

\caption{Comparison between the number of constraints in the original  and reduced linear program.
The first row indicates the number of elements of the Robinson spaces ($n$).
The second row  reports the average, in a sample of one hundred spaces, of the proportion of constraints remained after removing redundant constraints ($\hat{\mu}$).
Last row shows the coefficient of variation, computed as the sample standard deviation normalized by the average ($\hat{\sigma}/\hat{\mu}$). 
}
\label{tabla:lp2-lp3}
\end{table}

 A first type of redundant columns corresponds to constraints for which either the left center or the right center of the pair coincides with one of the indices of the pair (see Figures~\ref{fig:B_sctructute}(a) and~\ref{fig:B_sctructute}(b)). 
Let $(\mathbf{Y_S})_{ij}$ be a column of this type and suppose that the vector $\mathbf{y_N}$ is a solution of problem~\eqref{eq:lpS}. 

Observe that the column $(\mathbf{Y_S})_{ij}$ contains only non–positive entries. Hence, $(\mathbf{y_N})_{ij}$ cannot be the only nonzero component of $\mathbf{y_N}$, since otherwise we would have
\[
\mathbf{Y_S}\, \mathbf{y_N} = (\mathbf{Y_S})_{ij}\,(\mathbf{y_N})_{ij} < 0,
\]
which contradicts the condition $\mathbf{Y_S}\, \mathbf{y_N} \ge 0$ together with $\mathbf{y_N} \neq 0$. 

Therefore, setting $(\mathbf{y_N})_{ij}=0$ yields a vector $\mathbf{y_N'} \ge 0$ that is still nonzero and satisfies
\[
\mathbf{Y_S}\, \mathbf{y_N'} \;=\; \mathbf{Y_S}\, \mathbf{y_N} - (\mathbf{y_N})_{ij}\,(\mathbf{Y_S})_{ij}
\;\ge\; \mathbf{Y_S}\, \mathbf{y_N} 
\;\ge\; 0.
\]
Consequently, $(\mathbf{Y_S})_{ij}$ is redundant in the sense of the definition above.

 The second type of redundant columns corresponds to constraints associated with pairs $(i,j)$ with $i<j$ for which there exists another pair $(i',j')$, with $i'<j'$, such that
\[
\LC(i,j)=\LC(i',j'), \qquad i' \le i, \qquad j' \le j
\]
(see, for instance, the columns corresponding to $(2,5)$ and $(1,5)$ in Figure~\ref{fig:B_example}). 
In this situation, the column $(\mathbf{Y_S})_{ij}$ is component-wise dominated by the column $(\mathbf{Y_S})_{i'j'}$, that is,
\[
(\mathbf{Y_S})_{ij} \le (\mathbf{Y_S})_{i'j'}.
\]

Let $\mathbf{y_N}$ be a solution of \eqref{eq:lpS}. Define a new vector $\mathbf{y_N'}$ by setting
\[
(\mathbf{y_N'})_{ij}=0, 
\qquad 
(\mathbf{y_N'})_{i'j'}=(\mathbf{y_N})_{ij}+(\mathbf{y_N})_{i'j'},
\]
and leaving all other components unchanged. 
Clearly, $\mathbf{y_N'} \ge 0$ and $\mathbf{y_N'} \neq 0$. Moreover,
\[
\mathbf{Y_S} \mathbf{y_N'}
= \mathbf{Y_S} \mathbf{y_N} + \bigl((\mathbf{Y_S})_{i'j'} - (\mathbf{Y_S})_{ij}\bigr)(\mathbf{y_N})_{ij}
\;\ge\; \mathbf{Y_S} \mathbf{y_N}
\;\ge\; 0,
\]
where the inequality follows from the fact that $(\mathbf{Y_S})_{i'j'}-(\mathbf{Y_S})_{ij}$ is component-wise nonnegative. 
Hence, setting $(\mathbf{y_N})_{ij}=0$ does not affect feasibility, and the column $(\mathbf{Y_S})_{ij}$ is redundant.

An analogous argument shows that a column $(\mathbf{Y_S})_{ji}$ with $i<j$ is redundant if there exists another pair $(j',i')$ such that
\[
\RC(j,i)=\RC(j',i'), \qquad i' \ge i, \qquad j' \ge j
\]
In this case, the column $(\mathbf{Y_S})_{ji}$ is component-wise dominated by $(\mathbf{Y_S})_{j'i'}$, and the same redistribution argument applies.

\begin{figure}[t]
    \centering
 \includegraphics[]{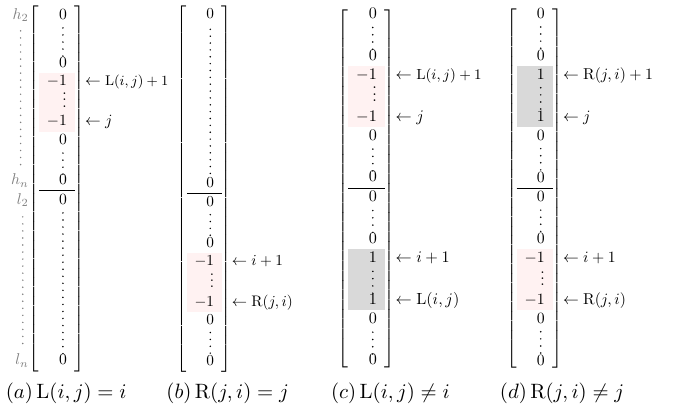}
    \caption{Structure types of columns in $\mathbf{Y_S}$.}
   \label{fig:B_sctructute}
\end{figure}

\label{sec:experimental}
  
Removing redundant triples significantly reduces the number of constraints in the linear program.
To evaluate this reduction, we conduct an experimental analysis comparing the number of constraints in the original matrix $\mathbf{Y_S}$ with those remaining after eliminating the redundant columns.
For each size in an increasing sequence of space dimensions, we use a random generator to sample one hundred distinct Robinson spaces.
For each sampled space, we compute the proportion of columns retained after removing redundancies.
  Table~\ref{tabla:lp2-lp3} presents the experimental results.
  The number of remaining constraints is approximately 20\% of the original count, regardless of space size.
  Moreover, the coefficient of variation ranges between  3-9\%, demonstrating that this reduction rate is highly consistent across all generated instances.
 
  \paragraph{Existence of Valid Drawings in Trees.}
  \label{sec:characterization-applications}
We now exploit the different characterizations of the existence of valid drawings in a caterpillar in order to identify subclasses of Robinson spaces for which the (non-)existence of such a drawing is guaranteed.

We begin by providing structural conditions that are sufficient to ensure the existence of a valid drawing in a caterpillar. 
In particular, we show that strict Robinson spaces whose left and right centers are suitably correlated with the midpoints of their corresponding intervals always admit a valid drawing in a caterpillar.

To prove this result, we first establish a structural property of the center matrix for strict Robinson spaces: the left and right centers of any pair are necessarily adjacent. Specifically, for every pair \(i < j\) in a strict Robinson space,
\[
\LC(i,j) = \RC(i,j) - 1.
\]
By definition, the right center \(\RC(i,j) = k\) is the smallest index in \(\{i,\ldots,j\}\) satisfying \(d(k,j) < d(k,i)\). 
Therefore, the preceding element $(k-1)$ must satisfies  \(d(k-1,j) \ge d(k-1,i)\), as otherwise this would contradict the minimality of \(\RC(i,j)\). 
Because the dissimilarities are strictly ordered (no ties), equality is impossible; therefore,
\[
d(k-1,j) > d(k-1,i).
\]
This inequality means that \(k-1\) is closer to \(i\) than to \(j\), thus \(k-1 = \LC(i,j)\), establishing the adjacency property.
This structural constraint has an important implication: the position of one center completely determines the other, reducing the degrees of freedom in the center matrix to a single family of values. This observation proves crucial in the proof of the following theorem.

  \begin{theorem}
    \label{thm:center-feasible}
    Let $(S,\dd)$ be a strict Robinson dissimilarity space. 
If for every pair $i<j$ we have
\[
\LC(i,j)\le \frac{i+j}{2} \le \RC(j,i),
\]
then $(S,\dd)$ admits a valid drawing in a caterpillar.
  \end{theorem}
\begin{proof}
We prove that for every vector $\mathbf{y_N}\ge 0$, the sum of the entries of $\mathbf{Y_S}\,\mathbf{y_N}$ is non-positive. 
In particular, this implies that $\mathbf{Y_S}\,\mathbf{y_N}$ cannot be a strictly positive vector.

For every pair $i<j$, let
\(
m_{ij}=(i+j)/2
\)
denote its midpoint. 
Consider a pair of elements $i<j-1$. Summing all the rows of the column $(\mathbf{Y_S})_{ij}$ corresponding to a left–center constraint (see Figure~\ref{fig:B_sctructute}(c)) yields
\begin{equation*}
\label{eq:S1}
\mathbf{S}_{ij}
=-(j-\LC(i,j))+(\LC(i,j)-i)
=2\bigl(\LC(i,j)-m_{ij}\bigr).
\end{equation*}
 Summing now all the rows of the column $(\mathbf{Y_S})_{ji}$ corresponding to a right–center constraint (see Figure~\ref{fig:B_sctructute}(d)) gives
\begin{equation*}
\label{eq:S2}
\mathbf{S}_{ji}
=(j-\RC(j,i))-(\RC(j,i)-i)
=2\bigl(m_{ij}-\RC(j,i)\bigr).
\end{equation*}
Hence, the sum of all the entries of the vector $\mathbf{Y_S}\,\mathbf{y_N}$ is
\begin{equation}\label{eq:sum-rows} 
\sum_{i<j-1}\!\!\!\left( \mathbf{S}_{ij}\,{\mathbf{y_N}}_{ij}+ \mathbf{S}_{ji}\,{\mathbf{y_N}}_{ji}\right) = 2\!\!\sum_{i<j-1}\!\!\!\Big((\LC(j,i) - m_{ij} ){\mathbf{y_N}}_{ij} + ( m_{ij} - \RC(j,i) ){\mathbf{y_N}}_{ji}\Big) \end{equation}

Under the hypothesis of the theorem, we have
\(
\LC(i,j)\le m_{ij}
\,\text{and}\,
m_{ij}\le \RC(j,i)
\)
for every $i<j$, and therefore both coefficients
\(
\LC(i,j)-m_{ij}
\,\text{and}\,
m_{ij}-\RC(j,i)
\)
are non-positive. Since $\mathbf{y_N}\ge 0$, every term in the sum \eqref{eq:sum-rows} is non-positive, and consequently the sum of the entries of $\mathbf{Y_S}\,\mathbf{y_N}$ is non-positive.

It follows that $\mathbf{Y_S}\,\mathbf{y_N}$ cannot be a non-negative vector. Hence, by Lemma~\ref{thm:kernel_characterization}, the strict Robinson space $(S,\dd)$ admits a valid drawing in a caterpillar.
\end{proof}

We now use the kernel characterization given in Lemma~\ref{thm:kernel_characterization} to identify simple structural situations in which the existence of a valid drawing in a caterpillar is guaranteed. In particular, we show that when the matrix $\mathbf{Y_S}$ has very few non-redundant columns, the feasibility problem~\eqref{eq:lpS} becomes highly constrained and cannot admit any nonzero nonnegative solution. This implies, by duality, that the original Robinson space necessarily admits a valid drawing. The next lemma formalizes this idea by proving that if $Y_S$ contains at most one non-redundant column of a type, then the space always has a valid drawing in a caterpillar.

  \begin{theorem}
       \label{thm:single-non-redundant}
Let $(S,\dd)$ be a strict Robinson space. 
If the matrix $\mathbf{Y_S}$ contains fewer than two non-redundant columns of the form $(\mathbf{Y_S})_{ij}$ with $i<j$ or of the form $(\mathbf{Y_S})_{ji}$ with $i<j$, then $(S,\dd)$ admits a valid drawing in a caterpillar.
\end{theorem}

\begin{proof}
First, observe that for every solution of (LP-S) there exists a solution where  $\mathbf{Y_S}\,\mathbf{y_N}$ can be written as a linear combination of its non-redundant columns:
\[
\mathbf{Y_S}\,\mathbf{y_N}
=
\sum_{\substack{i<j-1 \\ (i,j)\ \text{non-redundant}}} (\mathbf{Y_S})_{ij}\,(\mathbf{y_N})_{ij}
+
\sum_{\substack{i<j-1 \\ (j,i)\ \text{non-redundant}}} (\mathbf{Y_S})_{ji}\,(\mathbf{y_N})_{ji}.
\]

Without loss of generality, suppose that $\mathbf{Y_S}$ contains fewer than two non-redundant columns of the form $(\mathbf{Y_S})_{ij}$ with $i<j-1$. 
We distinguish two cases.

\paragraph{Case 1: No such columns exist.}
All non-redundant columns are of the form $(\mathbf{Y_S})_{ji}$ with $i<j$ (see Figure~\ref{fig:B_sctructute}(d)). 
In each such column, the rows associated with the variables $\patas_k$, for $k\in\{2,\ldots,n\}$, contain only non-positive entries, and at least one of them is equal to $-1$. 
Hence, for any $\mathbf{y_N}\ge 0$ with $\mathbf{y_N}\neq 0$, the vector $\mathbf{Y_S}\,\mathbf{y_N}$ must contain at least one negative entry.

\paragraph{Case 2: Exactly one such column exists.}
Suppose that $\mathbf{Y_S}$ contains a unique non-redundant column $(\mathbf{Y_S})_{\bar{\imath}\bar{\jmath}}$ indexed by a pair $(\bar{\imath},\bar{\jmath})$ with $\bar{\imath}<\bar{\jmath}$. 
We show that for every non-negative vector $\mathbf{y_N}$, all entries associated with constraints of pairs $(j,i)$ with $i>j$ must be zero, and therefore $\mathbf{Y_S}\,\mathbf{y_N}$ must contain a negative entry.

Let $(i,j)$ be a pair with $i<j$ such that $(\mathbf{y_N})_{ji}>0$. 
Then necessarily
\[
i\ge \bar{\imath}
\qquad\text{and}\qquad
\RC(j,i)\le \LC(\bar{\imath},\bar{\jmath}).
\]
Otherwise, there exists a row $\vertebras_r$ such that
\[
(\mathbf{Y_S})_{\vertebras_r,ji}=-1
\qquad\text{and}\qquad
(\mathbf{Y_S})_{\vertebras_r,\bar{\imath}\bar{\jmath}}=0
\]
(see Figures~\ref{fig:B_sctructute}(c) and~\ref{fig:B_sctructute}(d)),
which would imply
\[
(\mathbf{Y_S}\,\mathbf{y_N})_{\vertebras_r}<0,
\]
a contradiction.

Moreover, we must have $j<\bar{\jmath}$. 
Indeed, if $j\ge \bar{\jmath}$, then together with $i\ge \bar{\imath}$ we obtain
\[
\LC(j,i)<\RC(j,i)\le \LC(\bar{\imath},\bar{\jmath}),
\]
which contradicts the monotonicity of the center matrix.

Hence,
\[
(\mathbf{Y_S})_{\vertebras_{\bar{\jmath}},\bar{\imath}\bar{\jmath}}=-1
\qquad\text{and}\qquad
(\mathbf{Y_S})_{\vertebras_{\bar{\jmath}},ji}=0,
\]
and therefore
\[
(\mathbf{Y_S}\,\mathbf{y_N})_{\vertebras_{\bar{\jmath}}}
=
(\mathbf{Y_S})_{\vertebras_{\bar{\jmath}},\bar{\imath}\bar{\jmath}}\,(\mathbf{y_N})_{\bar{\imath}\bar{\jmath}}
+
\sum_{i<j}(\mathbf{Y_S})_{\vertebras_{\bar{\jmath}},ji}\,(\mathbf{y_N})_{ji}
=
-(\mathbf{y_N})_{\bar{\imath}\bar{\jmath}}<0.
\]
Thus, $\mathbf{Y_S}\,\mathbf{y_N}$ necessarily contains a negative entry.

In both cases, there is no nonzero vector $\mathbf{y_N}\ge 0$ such that $\mathbf{Y_S}\,\mathbf{y_N}\ge 0$. 
By Lemma~\ref{thm:kernel_characterization}, it follows that $(S,\dd)$ admits a valid drawing in a caterpillar.
\end{proof}

Using the reduction presented in Section~\ref{sec:strictrobinsonian}, together with Lemma~\ref{lema: relacion-dibujos-construccion} and the feasibility formulation given in Lemma~\ref{thm:kernel_characterization}, we performed an exhaustive analysis of Robinson spaces on small ground sets. This analysis shows that \emph{all} Robinson spaces on four and on five elements admit a valid drawing in a caterpillar. We emphasize that this result holds not only for strict Robinson spaces, but for Robinson spaces in general. In contrast, in the next theorem we exhibit a strict Robinson space on six elements that does not admit any valid drawing in a caterpillar. Consequently, by Corollary~\ref{thm:caterpillar}, this space does not admit a valid drawing in any tree.

\begin{figure}[t]
    \centering
 \includegraphics[]{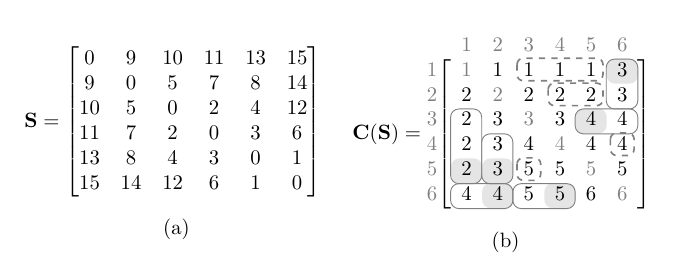}
\caption{ (a) A strict Robinson space $\mathbf{S}$ with no valid drawing in a tree represented by its dissimilarity matrix. (b) Matrix of centers $\mathbf{\centers(S)}$ of the space $\mathbf{S}$. Dashed blocks indicate redundant pairs of the first type. In each solid block, the highlighted pair is the unique non-redundant one and the only one that must be included in the feasibility formulation.
}
\label{fig:novalidtree}
\end{figure}

\begin{theorem}
\label{thm:novalidtree}
There exists a Robinson space on six elements that admits no valid drawing in any tree.
\end{theorem}

\begin{proof}
To prove the result, we exhibit a strict Robinson space $S$ on six elements that admits no valid drawing in any tree. 
Figure~\ref{fig:novalidtree} shows the dissimilarity matrix of $S$ together with its corresponding center matrix, where the non-redundant constraints are highlighted.

Therefore, it is sufficient to consider only the columns of $\mathbf{Y_S}$ associated with these non-redundant constraints. 
Denoting by $\mathbf{Y'_S}$ the reduced matrix obtained by keeping only such columns, we obtain the matrix shown in Figure \ref{fig:ys-contraejemplo}.

\begin{figure}[t]
    \centering
    \includegraphics[]{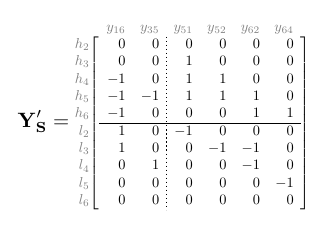}
    \caption{Matrix $\mathbf{Y'_S}$ of the strict Robinson space shown in Figure \ref{fig:novalidtree}.}
    \label{fig:ys-contraejemplo}
\end{figure}

Now define
\[
\mathbf{y'_N}=(1,1,1,0,1,0)^\top .
\]
A direct computation shows that
\[
\mathbf{Y'_S}\,\mathbf{y'_N}
=
(0,1,0,0,0,0,0,0,0,0)^\top,
\]
which is a nonzero non-negative vector. 
Therefore, by Lemma~\ref{thm:kernel_characterization}, the feasibility problem~\eqref{eq:lpS} admits a solution, and consequently the strict Robinson space $S$ does not admit a valid drawing in a caterpillar, and by Corollary \ref{thm:caterpillar} in any tree.
\end{proof}

\section{Conclusions}
\label{sec:conclusions}
In this work, we challenge the widespread belief that Robinson spaces are inherently linear from a metric point of view, a perception largely rooted in their combinatorial linear structure. We provide an algorithmic framework that decides whether a given Robinson space admits such a linear (tree-based) representation. Moreover, we exhibit an explicit example of a Robinson space that not only fails to admit a linear representation, but in fact does not admit any valid drawing in a real tree. This shows that the class of Robinson spaces is geometrically richer and more complex than previously expected.

Theorem~\ref{thm:novalidtree} proves that  there exist Robinson spaces that admit no valid drawing in any tree.
This negative result naturally opens a broader question. Since trees correspond to low-dimensional geometric structures, it is natural to ask whether Robinson spaces always admit valid drawings in higher-dimensional Euclidean spaces. More precisely, does there exist an integer $\ell$ such that every Robinson space on $n$ elements admits a valid drawing in $\mathbb{R}^\ell$? If so, is $\ell$ a constant, or must it depend on $n$? Answering this question would provide a deeper geometric understanding of Robinson spaces and suggests a promising direction for future research.

\section*{Declarations}

\paragraph{Funding} Mauricio Soto-Gomez was supported by National Center for Gene Therapy and Drugs Based on RNA Technology—MUR (Project no. CN 00000041) funded by NextGeneration EU program.

\paragraph{Conflict of interest} The authors have no competing interests to declare that are relevant to the content of this article.

\paragraph{Code availability} All the code used for the experimental results can be found at \url{https://github.com/Kurufo/Robinson_valid_drawing}.

\bibliography{sn-bibliography}
\end{document}